\documentclass[a4paper,11pt]{article}
\usepackage[multidot]{grffile}
\usepackage{latexsym,amssymb,enumerate,amsmath,epsfig,amsthm}
\usepackage[margin=1in]{geometry}
\usepackage{setspace,color}
\usepackage{graphics}
\usepackage{graphicx}
\usepackage[ruled]{algorithm2e}
\usepackage{empheq}
\usepackage{bm,amsmath}
\usepackage{subcaption}
\usepackage{hyperref} 
\hypersetup{
    colorlinks=true,
    linkcolor=blue,
    filecolor=magenta,      
    urlcolor=cyan,
    citecolor=red,
    pdfpagemode=FullScreen,
    }

\usepackage{rotating,booktabs,array,amsmath,amssymb}
\newcolumntype{C}{>{$\displaystyle}c<{$}}

\usepackage{multirow}
\newcommand\fixup{\kern-\fontcharic\scriptfont2`\"}

\usepackage{lineno}

\newcommand{\bS}{\mathbb{S}}
\newcommand{\bR}{\mathbb{R}}
\newcommand{\Log}{\operatorname{Log}}
\newcommand{\Exp}{\operatorname{Exp}}
\newcommand{\SLERP}{\operatorname{SLERP}}

\newtheorem{thm}{Theorem}
\newtheorem{prop}{Proposition}
\newtheorem{lemma}{Lemma}
\newtheorem{remark}{Remark}
\newtheorem{assump}{Assumption}
\numberwithin{equation}{section}

\title{Local Consistency and Higher-Order Structure of Spherical Interpolation}
\author{
Shingyu Leung\thanks{Department of Mathematics, the Hong Kong University of Science and Technology, Clear Water Bay, Hong Kong. Email: {\bf masyleung@ust.hk}}
}
\date{}

\begin{document}
\thispagestyle{plain}
\maketitle

\begin{abstract}
Spherical Interpolation of orDER \(n\) (SIDER-\(n\)) is a recursive
high-order interpolation construction for data on the unit sphere \(\bS^2\),
built from repeated spherical linear interpolation (SLERP).  This paper gives
a local consistency analysis of SIDER for smooth spherical curves sampled at
equally spaced parameter values.  The analysis is carried out in geodesic
normal coordinates, which allows the SIDER recursion to be compared with
classical Neville interpolation while retaining the curvature-dependent
corrections introduced by SLERP.  We first derive local expansions of SLERP and
show that SIDER2 has third-order accuracy; its leading error has the same
shifted nodal structure as Euclidean quadratic interpolation.  We then prove
that the adjacent SIDER2 errors entering SIDER3 have a common leading
coefficient, so that the SIDER3 recurrence cancels the cubic term and yields
fourth-order accuracy.  Carrying the expansion one order further gives the
corresponding coefficient compatibility for SIDER3 and proves fifth-order
accuracy of SIDER4.  Finally, we introduce a degree-filtered formal expansion
framework for the general SIDER recursion.  This framework proves that, for
each fixed \(n\), SIDER-\(n\) preserves the required polynomial degree
structure in the normalized stencil variable.  Together with the interpolation
conditions at the \(n+1\) nodes, this yields the local consistency estimate
$d_{\bS^2}\bigl(\gamma(\theta h),P_i^{[n]}(\theta;h)\bigr)
        =O(h^{n+1})$ under the stated smoothness and small-stencil assumptions.
\end{abstract}

\noindent\textbf{Keywords.} Spherical interpolation; SLERP; SIDER; Taylor expansion;
normal coordinates; Neville interpolation; unit sphere.

\section{Introduction}

Interpolation of data constrained to a nonlinear manifold is a fundamental
problem in scientific computing.  In many applications, the quantities to be
interpolated are not arbitrary Euclidean vectors, but normalized directions,
orientations, rotations, or other geometric states satisfying nonlinear
constraints.  A particularly important example is interpolation on the unit
sphere \(\bS^2\).  If spherical data are interpolated by standard Euclidean
polynomials, the resulting curve generally leaves the sphere; projecting the
curve back to \(\bS^2\) restores the constraint but may change both the
geometry and the approximation properties of the reconstruction.  This
motivates interpolation methods that respect the spherical constraint
throughout the construction.

Spherical interpolation appears in computer graphics, rigid-body kinematics,
attitude dynamics, molecular simulation, fluid mechanics, and numerical methods
involving quaternion- or direction-valued unknowns.  Quaternions provide a
convenient representation of rotations and have been widely used in orientation
interpolation and rotation sequences; see, for example,
\cite{ham63,kui02,muk02,zha97}.  The classical spherical linear interpolation
method, or SLERP, joins two points on \(\bS^2\) by a geodesic with constant
angular speed \cite{sho85}.  SLERP is geometrically natural, preserves the
unit-length constraint exactly, and is easy to implement.  However, piecewise
SLERP has limited smoothness at interpolation nodes, and therefore does not
provide a systematic high-order interpolation framework.

In \cite{fonleu23}, we introduced a recursive family of spherical
interpolants called SIDER, standing for Spherical Interpolation of orDER
\(n\).  The construction is inspired by Bezier and Neville-type interpolation,
but replaces affine interpolation by SLERP.  SIDER2 interpolates three
spherical data points and plays the role of a quadratic reconstruction; SIDER3
combines two adjacent SIDER2 interpolants and is analogous to a cubic
reconstruction.  More generally, SIDER-\(n\) is obtained recursively from two
adjacent SIDER-\((n-1)\) interpolants.  This construction provides a natural
high-order interpolation framework on \(\bS^2\), and it also serves as the
building block of the spherical essentially non-oscillatory method, SENO,
introduced in \cite{fonleu23}.  Numerical experiments in that work show
the expected high-order behavior of SIDER and improved non-oscillatory
behavior of SENO near nonsmooth data.

The purpose of the present paper is to provide a local consistency analysis of
the SIDER construction for smooth curves on \(\bS^2\).  The analysis is
carried out in geodesic normal coordinates centered at the exact value of the
curve.  In these coordinates, the SIDER recursion can be compared with
classical Neville interpolation, while the curvature-dependent corrections
introduced by SLERP remain visible.  This point is important: although SLERP is
locally affine to leading order, its first nonlinear correction appears at
cubic order.  Since this is the same order as the leading error of a quadratic
reconstruction, a rigorous proof of fourth-order accuracy for SIDER3 cannot
treat the nonlinear part of SLERP as an irrelevant remainder.

We first derive local normal-coordinate expansions of SLERP and use them to
prove a refined third-order error formula for SIDER2.  The leading error has
the same shifted nodal structure as ordinary quadratic interpolation.  We then
show that the two adjacent SIDER2 reconstructions entering SIDER3 have a common
leading cubic coefficient, so that the SIDER3 recurrence cancels this term and
yields fourth-order local accuracy.  By carrying the expansion one order
further, we obtain the corresponding coefficient compatibility for SIDER3 and
prove fifth-order local accuracy of SIDER4.  Finally, we introduce a
degree-filtered formal expansion framework for the general SIDER recursion.
This framework proves that, for every fixed \(n\), SIDER-\(n\) achieves the
local consistency estimate \(O(h^{n+1})\) under the stated smoothness and
small-stencil assumptions.

The results are local and asymptotic.  We assume that the sampled points and
the intermediate values generated by the interpolation process remain in a
common geodesically convex neighborhood of \(\bS^2\).  This ensures that the
logarithm map is single-valued, the relevant geodesics are unique, and the
Taylor expansions are uniform.  These assumptions are natural for a local
consistency analysis and are consistent with the restrictions needed to avoid
antipodal ambiguity and large angular jumps in spherical interpolation.

The paper is organized as follows.  Section~2 reviews the geometric notation,
SLERP, and the SIDER construction.  Section~3 develops the local
normal-coordinate expansions used in the analysis.  Sections~4 and~5 prove the
third- and fourth-order accuracy of SIDER2 and SIDER3, respectively.
Section~6 extends the refined expansion to obtain the SIDER4 result.
Section~7 discusses the general higher-order recurrence and coefficient
compatibility, and Section~8 proves the all-order local consistency result
using a degree-filtered expansion framework.  The final section summarizes the
main conclusions and outlines possible extensions.

\section{Background}

This section collects the geometric notation and interpolation constructions
used throughout the paper.  The presentation is intrinsic to the unit sphere
\(\bS^2\), but it is consistent with the quaternion formulation commonly used
for rotations and orientation interpolation; see, for example,
\cite{ham63,kui02,muk02,zha97}.  The SIDER construction itself follows
\cite{fonleu23}.  Our purpose here is not to introduce a new algorithm,
but to record the precise form of the operations whose local Taylor expansions
will be used in the proof.

\subsection{Geometry of the unit sphere}

We regard the unit sphere as the embedded Riemannian manifold
\[
        \bS^2=\{x\in\bR^3: |x|=1\},
\]
equipped with the metric induced from the Euclidean inner product.  For
\(y\in\bS^2\), the tangent space is
\[
        T_y\bS^2=\{v\in\bR^3: v\cdot y=0\}.
\]
If \(v\in T_y\bS^2\) and \(r=|v|\), the exponential map on \(\bS^2\) is
\begin{equation}
\label{eq:exp_sphere}
        \Exp_y(v)
        =
        \cos r\, y+\frac{\sin r}{r}\,v ,
\end{equation}
with the usual interpretation \((\sin r/r)v=v\) at \(r=0\).  Conversely, if
\(z\in\bS^2\) is not antipodal to \(y\), let
\[
        \alpha=d_{\bS^2}(y,z)=\arccos(y\cdot z).
\]
Then the logarithm map is
\begin{equation}
\label{eq:log_sphere}
        \Log_y(z)
        =
        \frac{\alpha}{\sin\alpha}\bigl(z-(y\cdot z)y\bigr)
        \in T_y\bS^2 .
\end{equation}
Thus \(\Exp_y(\Log_y z)=z\) whenever \(z\) lies in the normal neighborhood of
\(y\).

The analysis below is local.  We will repeatedly fix an exact point
\(y=\gamma(s)\) on a smooth curve and represent nearby data by normal
coordinates \(x_i=\Log_y(p_i)\).  In this coordinate system, the exact point
\(\gamma(s)\) is represented by the origin, while nearby sampled values are
\(O(h)\).  This is the setting in which the Taylor expansions in later
sections are carried out.

\begin{assump}[Local stencil condition]
\label{assump:local_stencil}
Throughout the consistency analysis, the sampling step \(h\) is assumed small
enough that all points generated by the interpolation construction under
consideration lie in a common geodesically convex ball of radius \(O(h)\).
In particular, the logarithm map is single-valued, the relevant geodesics are
unique, and all extrapolations used to form SIDER control points remain inside
the same local neighborhood.
\end{assump}

This assumption is stronger than what may be required in practical
implementations, but it is natural for local error analysis.  It excludes
antipodal ambiguity and large angular separations, which are also problematic
for quaternion-based spherical interpolation.  Under this hypothesis, all
constants in the \(O(\cdot)\) estimates below may be chosen uniformly for
evaluation parameters in any fixed bounded interval.

\subsection{Spherical linear interpolation}

Let \(a,b\in\bS^2\) be non-antipodal points and define
\[
        \alpha=d_{\bS^2}(a,b)=\arccos(a\cdot b),\qquad 0<\alpha<\pi .
\]
The spherical linear interpolation from \(a\) to \(b\) is
\begin{equation}
\label{eq:slerp}
\SLERP(a,b,\lambda)
=
\frac{\sin((1-\lambda)\alpha)}{\sin\alpha}\,a
+
\frac{\sin(\lambda\alpha)}{\sin\alpha}\,b .
\end{equation}
For \(0\leq\lambda\leq1\), this is precisely the shorter geodesic segment
from \(a\) to \(b\), parameterized at constant angular speed.  Equivalently,
\[
        \SLERP(a,b,\lambda)
        =
        \Exp_a\!\bigl(\lambda\,\Log_a b\bigr).
\]
The formula remains meaningful for \(\lambda\) outside the interval
\([0,1]\), as long as the same geodesic branch is used.  In this case SLERP is
a geodesic extrapolation rather than an interpolation.  Such extrapolation is
an essential ingredient in SIDER2, where spherical control points are obtained
by extending a geodesic beyond a data point.

SLERP is the correct spherical analogue of affine interpolation between two
points.  If \(a\) and \(b\) are close to a fixed base point \(y\), and
\(u=\Log_y(a)\), \(v=\Log_y(b)\), then
\[
        \Log_y\bigl(\SLERP(a,b,\lambda)\bigr)
        =
        (1-\lambda)u+\lambda v+\text{higher-order terms}.
\]
The higher-order terms are curvature-dependent and are not identically zero in
general.  This observation is central to the present paper: although SIDER has
the same recursive shape as Neville interpolation, replacing affine
interpolation by SLERP introduces nonlinear geometric corrections.  The proof
of accuracy must show that these corrections either vanish at the relevant
order or are absorbed into higher-order remainders.

\subsection{SIDER2: a spherical quadratic interpolant}

The first nontrivial SIDER construction uses three equally spaced data points
\(p_0,p_1,p_2\in\bS^2\).  In Euclidean quadratic interpolation, one may view
the midpoint data value as being enforced by choosing suitable control points
in a quadratic Bezier construction.  SIDER2 follows the same geometric idea,
but replaces Euclidean line segments by spherical geodesic segments.

Define the two extrapolated spherical control points
\begin{equation}
\label{eq:sider2_controls}
        d_a=\SLERP(p_2,p_1,2),
        \qquad
        d_b=\SLERP(p_0,p_1,2).
\end{equation}
The point \(d_a\) is obtained by starting at \(p_2\), moving along the geodesic
toward \(p_1\), and continuing the same distance beyond \(p_1\).  Similarly,
\(d_b\) is obtained by starting at \(p_0\), moving toward \(p_1\), and
continuing beyond \(p_1\).  These control points are not arbitrary Bezier
control points; they are chosen so that the final curve interpolates the
middle data value \(p_1\).

For \(s\in[0,2h]\), set
\[
        \tau=\frac{s}{2h}.
\]
The SIDER2 interpolant associated with the stencil
\(\{p_0,p_1,p_2\}\) is
\begin{equation}
\label{eq:sider2}
P_{012}(s)
=
\SLERP\left(
        \SLERP(p_0,d_a,\tau),
        \SLERP(d_b,p_2,\tau),
        \tau
       \right).
\end{equation}
This formula is the spherical counterpart of a quadratic Bezier evaluation.
The two inner SLERP operations form two geodesic curves from \(p_0\) to
\(d_a\) and from \(d_b\) to \(p_2\), respectively.  The outer SLERP then
interpolates between the two intermediate points using the same parameter
\(\tau\).  Because every operation is a SLERP, the reconstructed value remains
on \(\bS^2\).

The special choice of \(d_a\) and \(d_b\) gives
\[
        P_{012}(0)=p_0,\qquad
        P_{012}(h)=p_1,\qquad
        P_{012}(2h)=p_2 .
\]
Therefore SIDER2 is a genuine interpolation formula through all three data
points.  This distinction is important: in a standard Bezier curve, intermediate
points are usually control points and need not lie on the curve; in SIDER2, the
control points are auxiliary spherical extrapolants introduced solely to force
interpolation of \(p_1\).

Although \eqref{eq:sider2} is naturally evaluated for \(0\leq s\leq2h\), the
recursive SIDER construction also evaluates lower-order SIDER pieces at
bounded extrapolation parameters.  For example, when SIDER2 is used inside
SIDER3, the left SIDER2 candidate built from \(p_0,p_1,p_2\) is evaluated for
\(0\leq s\leq3h\).  This is why later estimates for SIDER2 are stated
uniformly on fixed bounded intervals of the normalized parameter
\(\theta=s/h\), rather than only on the original interpolation interval.

\subsection{SIDER3 and the recursive construction}

SIDER3 is obtained by combining two adjacent SIDER2 reconstructions.  Let
\[
        p_i=\gamma(ih),\qquad i=0,1,2,3,
\]
and let \(s=\theta h\).  Denote by \(P_{012}(s)\) the SIDER2 reconstruction
from \(p_0,p_1,p_2\), and by \(P_{123}(s)\) the SIDER2 reconstruction from
\(p_1,p_2,p_3\).  Each candidate is evaluated at the same physical parameter
\(s\), using its own local SIDER2 parameter.  The SIDER3 interpolant is
\begin{equation}
\label{eq:sider3}
P_{0123}(s)
=
\SLERP\left(
        P_{012}(s),
        P_{123}(s),
        \frac{\theta}{3}
       \right),
        \qquad 0\leq \theta\leq 3 .
\end{equation}
The weight \(\theta/3\) is the same weight that appears in the equally spaced
Neville recursion for cubic interpolation.  It ensures that the left
candidate is selected at \(\theta=0\), the right candidate is selected at
\(\theta=3\), and the two candidates are blended in a manner consistent with
the location of the evaluation point in the four-point stencil.

More generally, for \(k\geq3\), let \(P_i^{[k]}(\theta;h)\) denote the
SIDER-\(k\) reconstruction associated with the stencil
\[
        p_i,p_{i+1},\ldots,p_{i+k},
        \qquad p_j=\gamma(jh),
\]
evaluated at \(s=\theta h\).  The recursive part of SIDER can be written as
\begin{equation}
\label{eq:sider_recursion_background}
P_i^{[k]}(\theta;h)
=
\SLERP\!\left(
        P_i^{[k-1]}(\theta;h),
        P_{i+1}^{[k-1]}(\theta;h),
        \frac{\theta-i}{k}
       \right).
\end{equation}
The case \(k=2\) is the special SIDER2 construction above, because the
quadratic interpolant requires the extrapolated control points
\eqref{eq:sider2_controls}.  For \(k\geq3\), the recursion has the same
structure as Neville interpolation, with SLERP replacing affine interpolation.

If all SLERP operations in \eqref{eq:sider2} and
\eqref{eq:sider_recursion_background} were replaced by ordinary affine
interpolation in a vector space, the resulting construction would be classical
polynomial interpolation on equally spaced nodes.  In that Euclidean setting,
the increase of order follows from the cancellation of leading nodal error
polynomials.  The spherical construction inherits this algebraic mechanism,
but only after one verifies that the leading geometric error coefficients have
the required compatibility across adjacent stencils.  The rest of the paper is
devoted to precisely this issue for SIDER2 and SIDER3, and to formulating the
corresponding condition for higher-order recursions.

\section{Local normal-coordinate expansions}

The purpose of this section is to establish the local analytic tools needed for
the accuracy proofs in the following sections.  The SIDER construction is
defined globally in terms of SLERP operations on \(\bS^2\), but its local
accuracy is most naturally studied in geodesic normal coordinates.  In such
coordinates, the exact spherical curve is represented by a vector-valued
function in a tangent plane, and SLERP can be compared with ordinary affine
interpolation.  This comparison is the bridge between the geometry of the
sphere and the algebraic cancellation mechanism familiar from Euclidean
Neville interpolation.

There is, however, a subtle point.  It is not sufficient to use only the
rough estimate
\[
        \Log_y(\SLERP(A,B,\lambda))
        =
        (1-\lambda)\Log_y(A)+\lambda\Log_y(B)+O(h^3),
\]
when \(A\) and \(B\) are \(O(h)\)-close to \(y\).  For SIDER2, such a statement
is enough to obtain a third-order consistency estimate.  For SIDER3, however,
the leading error of each SIDER2 candidate is itself of order \(h^3\).  Hence
the cubic non-affine correction in SLERP lies exactly at the order where
cancellation must be verified.  A rigorous fourth-order proof for SIDER3
therefore requires identifying the structure of this cubic term, rather than
absorbing it into an unspecified \(O(h^3)\) remainder.

The estimates below serve three roles.  First, we record the Taylor expansion
of a smooth curve on \(\bS^2\) in normal coordinates centered at the exact
evaluation point.  Second, we derive a cubic expansion of SLERP in those
coordinates.  This expansion shows explicitly how curvature modifies affine
interpolation at the first nontrivial order.  Third, we state a higher-order
local expansion and a perturbative consequence that will be used to justify
the final SLERP step in recursive SIDER constructions.  All statements are
local and uniform under the small-stencil assumption introduced in
Assumption~\ref{assump:local_stencil}.

Let \(\gamma:[0,nh]\to\bS^2\) be a smooth curve, and let
\[
        p_i=\gamma(ih).
\]
For a fixed evaluation point \(s=\theta h\), set
\[
        y=\gamma(s).
\]
We work in geodesic normal coordinates at \(y\), writing
\[
        x_i=\Log_y(p_i)\in T_y\bS^2 .
\]
For \(h\) sufficiently small, all \(x_i\) are \(O(h)\), uniformly for
\(\theta\) in a fixed bounded interval.  Equivalently, if
\[
        X(\eta)=\Log_y(\gamma(s+\eta)),
\]
then \(X(0)=0\), and the samples are obtained by taking
\[
        x_i=X((i-\theta)h).
\]
Taylor expansion gives
\begin{equation}
\label{eq:curve_taylor}
x_i
=
(i-\theta)h V
+\frac{(i-\theta)^2h^2}{2}A
+\frac{(i-\theta)^3h^3}{6}B
+\frac{(i-\theta)^4h^4}{24}D
+O(h^5),
\end{equation}
where \(V,A,B,D\in T_y\bS^2\) are bounded vectors depending on the derivatives
of \(\gamma\) at \(s\).  The first two coefficients agree with the velocity and
covariant acceleration at \(s\), because the Christoffel symbols vanish at the
origin of normal coordinates.  Higher coefficients contain covariant
derivatives of \(\gamma\) together with curvature contributions.  Their
explicit form is not needed here; only the existence of the expansion and the
uniform boundedness of the coefficients are used.

More generally, if \(\gamma\in C^{m+1}\), then the same argument gives
\[
        x_i
        =
        \sum_{\ell=1}^{m}
        \frac{(i-\theta)^\ell h^\ell}{\ell!}X^{(\ell)}(0)
        +O(h^{m+1}),
\]
with a remainder uniform for \(\theta\) in any fixed bounded interval.  This
is the normal-coordinate analogue of the ordinary Taylor expansion used in
Euclidean interpolation theory.

\subsection{Cubic expansion of SLERP}

We next expand SLERP in normal coordinates.  This is the main technical input
for the SIDER2 and SIDER3 error estimates.

\begin{lemma}[Cubic SLERP expansion in normal coordinates]
\label{lem:cubic_slerp}
Let \(y\in\bS^2\), and let \(A,B\) be sufficiently close to \(y\).  Put
\[
        U=\Log_y(A),\qquad W=\Log_y(B).
\]
For \(\lambda\) in a fixed bounded interval,
\begin{equation}
\label{eq:cubic_slerp}
\Log_y\!\left(\SLERP(A,B,\lambda)\right)
=
(1-\lambda)U+\lambda W+\mathcal C_\lambda(U,W)+O(\rho^5),
\end{equation}
where
\[
        \rho=|U|+|W|,
        \qquad
        Z=(1-\lambda)U+\lambda W,
\]
and
\begin{align}
\label{eq:cubic_term}
\mathcal C_\lambda(U,W)
= \frac{1}{6}\Big\{&
|Z|^2Z
+(1-\lambda)\lambda(2-\lambda)|U-W|^2U
+\lambda(1-\lambda^2)|U-W|^2W     \nonumber\\
&-(1-\lambda)|U|^2U
-\lambda |W|^2W
\Big\}.
\end{align}
In particular, the first non-affine correction is cubic in the endpoint
coordinates.
\end{lemma}

\begin{proof}
Rotate coordinates so that \(y\) is the north pole.  Then every tangent vector
\(U\in T_y\bS^2\) can be identified with a vector in \(\bR^2\), and
\[
        \Exp_y(U)
        =
        \left(\frac{\sin r}{r}U,\cos r\right),
        \qquad r=|U|.
\]
Let
\[
        A=\Exp_y(U),\qquad B=\Exp_y(W),
\]
and let
\[
        \alpha=d_{\bS^2}(A,B).
\]
The sine formula for SLERP gives
\[
        \SLERP(A,B,\lambda)
        =
        a_\lambda A+b_\lambda B,
\]
where
\[
        a_\lambda=\frac{\sin((1-\lambda)\alpha)}{\sin\alpha},
        \qquad
        b_\lambda=\frac{\sin(\lambda\alpha)}{\sin\alpha}.
\]
For small \(\alpha\),
\[
        a_\lambda
        =
        (1-\lambda)
        +
        \frac{(1-\lambda)\{1-(1-\lambda)^2\}}{6}\alpha^2
        +O(\alpha^4),
\]
and
\[
        b_\lambda
        =
        \lambda
        +
        \frac{\lambda(1-\lambda^2)}{6}\alpha^2
        +O(\alpha^4).
\]
Also,
\[
        \frac{\sin |U|}{|U|}
        =
        1-\frac{|U|^2}{6}+O(|U|^4),
        \qquad
        \frac{\sin |W|}{|W|}
        =
        1-\frac{|W|^2}{6}+O(|W|^4),
\]
and
\[
        \alpha^2=|U-W|^2+O(\rho^4).
\]
Therefore, the tangent component of the SLERP point is
\[
\begin{aligned}
        \Pi_T \SLERP(A,B,\lambda)
        &=
        (1-\lambda)U+\lambda W                                      \\
        &\quad
        +\frac{1}{6}
        \Big[
        (1-\lambda)\lambda(2-\lambda)|U-W|^2U
        +\lambda(1-\lambda^2)|U-W|^2W                                \\
        &\qquad\qquad
        -(1-\lambda)|U|^2U-\lambda |W|^2W
        \Big]
        +O(\rho^5),
\end{aligned}
\]
where \(\Pi_T\) denotes projection onto the first two, tangent, coordinates.

It remains to apply the logarithm map at \(y\).  If a point on the sphere is
written as \((\xi,\sqrt{1-|\xi|^2})\), then
\[
        \Log_y(\xi,\sqrt{1-|\xi|^2})
        =
        \frac{\arcsin|\xi|}{|\xi|}\xi
        =
        \xi+\frac{|\xi|^2}{6}\xi+O(|\xi|^5).
\]
Here the leading tangent component is
\[
        Z=(1-\lambda)U+\lambda W.
\]
Thus applying the logarithm adds the additional cubic term
\[
        \frac{1}{6}|Z|^2Z.
\]
Combining the tangent-component expansion with this logarithm expansion gives
\eqref{eq:cubic_slerp}--\eqref{eq:cubic_term}.  The remainder is uniform for
\(\lambda\) in bounded intervals because all expansions are taken in a fixed
normal neighborhood and all coefficients are smooth there.
\end{proof}

\begin{remark}[Why the cubic term matters]
The estimate
\[
\Log_y(\SLERP(A,B,\lambda))
=
(1-\lambda)U+\lambda W+O(\rho^3)
\]
is often sufficient for informal consistency arguments.  In the present
problem it is not sufficient for proving fourth-order accuracy of SIDER3.
Indeed, the two SIDER2 candidates used in SIDER3 have leading errors of order
\(h^3\).  If the cubic part of SLERP were left unspecified, one could not
verify that the two leading \(h^3\) contributions cancel.  Lemma
\ref{lem:cubic_slerp} identifies the cubic contribution explicitly, allowing
the SIDER2 error to be expanded with the precision needed for the SIDER3
argument.
\end{remark}

The cubic term in Lemma~\ref{lem:cubic_slerp} is not an artifact of the proof.
It is a genuine curvature correction.  For example, take
\(y=(0,0,1)\),
\[
        U=(\varepsilon,0),\qquad W=(0,\varepsilon),
        \qquad \lambda=\frac12 .
\]
A direct substitution into \eqref{eq:cubic_slerp} gives
\[
\Log_y\!\left(
        \SLERP(\Exp_y U,\Exp_y W,1/2)
        \right)
=
\left(
        \frac{\varepsilon}{2}+\frac{\varepsilon^3}{12},
        \frac{\varepsilon}{2}+\frac{\varepsilon^3}{12}
\right)
+O(\varepsilon^5).
\]
Thus the midpoint obtained by SLERP is not the affine midpoint in normal
coordinates up to fourth order.  The first discrepancy is cubic.

\subsection{A high-order SLERP expansion}

The cubic expansion is enough for the detailed SIDER2 and SIDER3 proofs.
For later discussion of the recursive SIDER structure, it is useful to record
a higher-order version.  The next lemma does not require the explicit
coefficients of all higher-order terms.  It only states that such an expansion
exists, that the expansion begins with the affine term, and that there is no
quadratic correction.

\begin{lemma}[High-order local expansion of SLERP]
\label{lem:high_order_slerp}
Fix an integer \(M\geq1\) and a compact set of base points \(y\in\bS^2\).
There exist \(\rho_0>0\) and \(C_M>0\) such that, for
\[
        |U|+|W|\leq\rho_0
\]
and for \(\lambda\) in any fixed bounded interval, the map
\[
G_y(U,W,\lambda)
:=
\Log_y\!\left(\SLERP(\Exp_y U,\Exp_y W,\lambda)\right)
\]
has the finite expansion
\begin{equation}
\label{eq:high_order_slerp}
G_y(U,W,\lambda)
=
(1-\lambda)U+\lambda W
+
\sum_{r=3}^{M}\mathcal G_r(y;U,W,\lambda)
+
\mathcal R_{M+1}(y;U,W,\lambda),
\end{equation}
where each \(\mathcal G_r\) is homogeneous of degree \(r\) in \((U,W)\),
smooth in \(y\), and smooth in \(\lambda\).  For bounded \(\lambda\), the
remainder satisfies
\[
        |\mathcal R_{M+1}(y;U,W,\lambda)|
        \leq
        C_M(|U|+|W|)^{M+1}.
\]
Moreover, there is no quadratic term in the expansion.
\end{lemma}

\begin{proof}
Again rotate coordinates so that \(y\) is the north pole.  In geodesic normal
coordinates,
\[
        \Exp_y U
        =
        \left(\frac{\sin |U|}{|U|}U,\cos |U|\right).
\]
For \(A=\Exp_yU\) and \(B=\Exp_yW\), SLERP is given by
\[
        \SLERP(A,B,\lambda)
        =
        \frac{\sin((1-\lambda)\alpha)}{\sin\alpha}A
        +
        \frac{\sin(\lambda\alpha)}{\sin\alpha}B,
        \qquad
        \alpha=d_{\bS^2}(A,B).
\]
In a sufficiently small normal neighborhood, the squared distance
\(\alpha^2\) is a smooth function of \((U,W)\).  The sine ratios can be written
as
\[
        \frac{\sin(\lambda\alpha)}{\sin\alpha}
        =
        \lambda
        \frac{\operatorname{sinc}(\lambda\alpha)}
             {\operatorname{sinc}(\alpha)},
        \qquad
        \operatorname{sinc} z=\frac{\sin z}{z},
\]
and similarly for the factor containing \(1-\lambda\).  Since
\(\operatorname{sinc} z\) is an even analytic function of \(z\), these ratios
are smooth functions of \(\alpha^2\) and \(\lambda\) for \(\alpha\) small.
Consequently, the Euclidean coordinates of
\(\SLERP(\Exp_yU,\Exp_yW,\lambda)\) are smooth functions of
\((U,W,\lambda)\).

The logarithm map at \(y\) is also smooth in the same normal neighborhood.
Therefore \(G_y(U,W,\lambda)\) is a smooth function of \((U,W,\lambda)\) near
\((0,0,\lambda)\).  Taylor's theorem in the variables \((U,W)\) gives
\eqref{eq:high_order_slerp}, with a uniform remainder for \(y\) in a compact
set and \(\lambda\) in a bounded interval.

It remains to justify the absence of a quadratic term.  This follows either
from the explicit cubic computation in Lemma~\ref{lem:cubic_slerp}, or directly
from the structure of the above formula.  The expansions of
\(\sin |U|/|U|\), \(\cos|U|\), and the sine ratios involve only even powers of
\(|U|\), \(|W|\), and \(\alpha\).  Since the tangent component is multiplied by
\(U\) or \(W\), the first possible correction beyond the linear term is cubic.
The logarithm map in normal coordinates has the form
\[
        \xi\mapsto \xi+\frac{|\xi|^2}{6}\xi+O(|\xi|^5),
\]
which also contributes no quadratic term.  Hence the quadratic homogeneous
part vanishes.
\end{proof}

\subsection{Affine stability under high-order perturbations}

The final lemma records a simple but important consequence of the previous
expansion.  In the recursive SIDER construction, after a lower-order
interpolant has already approximated the exact point to order \(h^m\), the
final SLERP is applied to two points whose logarithmic errors are
\(O(h^m)\), not \(O(h)\).  In that situation the nonlinear part of SLERP is
much smaller than the leading affine combination.

\begin{lemma}[Affine stability for high-order perturbations]
\label{lem:affine_stability}
Let \(U(h),W(h)\in T_y\bS^2\) satisfy
\[
        U(h)=O(h^m),\qquad W(h)=O(h^m),
        \qquad m\geq2.
\]
Then, uniformly for bounded \(\lambda\),
\[
\Log_y\!\left(
\SLERP(\Exp_y U(h),\Exp_y W(h),\lambda)
\right)
=
(1-\lambda)U(h)+\lambda W(h)+O(h^{3m}).
\]
In particular, since \(3m\geq m+2\) for \(m\geq1\), the non-affine part of
this SLERP does not affect the terms through order \(h^{m+1}\).
\end{lemma}

\begin{proof}
Apply Lemma~\ref{lem:high_order_slerp} with \(U=U(h)\) and \(W=W(h)\).  The
linear part is
\[
        (1-\lambda)U(h)+\lambda W(h).
\]
There is no quadratic term, and the first nonlinear term is cubic in
\((U(h),W(h))\).  Since \(U(h)\) and \(W(h)\) are both \(O(h^m)\), this cubic
term is \(O(h^{3m})\).  The same bound applies to the remainder if the
expansion is taken to sufficiently high order.  This proves the claim.
\end{proof}

The role of Lemma~\ref{lem:affine_stability} is different from that of
Lemma~\ref{lem:cubic_slerp}.  In the SIDER2 construction, SLERP is applied to
points that are \(O(h)\)-close to the exact curve, so the cubic correction is
of order \(h^3\) and must be retained.  In the final SIDER3 blending step,
however, SLERP is applied to two SIDER2 approximations that are already
\(O(h^3)\)-close to the exact value.  The nonlinear part of that final SLERP
is therefore \(O(h^9)\), and the leading behavior is simply the affine
combination of the two SIDER2 errors.  This separation of roles is the key
technical reason why the Neville-type cancellation can be carried out
rigorously for SIDER3.

\section{Accuracy of SIDER2}

We begin the accuracy analysis with SIDER2, since it is the first genuinely
nonlinear member of the SIDER family and also the basic building block for the
SIDER3 construction.  The goal of this section is twofold.  First, we prove
that SIDER2 is locally third-order accurate for smooth curves on \(\bS^2\).
Second, we record the leading error in a form that is precise enough to be
used in the SIDER3 proof.  In particular, it is not enough to know only that
the error is \(O(h^3)\); for the next recursive step, we need to know that the
leading cubic term has the same nodal-factor structure as the error of
ordinary quadratic interpolation.

Let
\[
        p_i=\gamma(ih),\qquad i=0,1,2,
\]
and let \(P_{012}\) denote the SIDER2 interpolant defined in
\eqref{eq:sider2}.  For an evaluation point \(s=\theta h\), we work in normal
coordinates centered at the exact point
\[
        y=\gamma(s).
\]
Thus
\[
        x_i=\Log_y(p_i),\qquad i=0,1,2,
\]
and the exact value \(\gamma(s)\) corresponds to the origin of
\(T_y\bS^2\).  The SIDER2 error is therefore represented by the single vector
\[
        \Log_{\gamma(s)}P_{012}(s).
\]
The proposition below gives its leading Taylor expansion.

\begin{prop}[Refined local error of SIDER2]
\label{prop:sider2_refined}
Let \(K\subset\bR\) be a fixed bounded interval.  Let \(\gamma\in C^4\) on a
neighborhood of \(\{\theta h:\theta\in K\}\), and assume that all points
appearing in the SIDER2 construction remain in a common geodesically convex
ball of radius \(O(h)\).  Let \(P_{012}\) be the SIDER2 interpolant
\eqref{eq:sider2}.  For \(s=\theta h\), uniformly for \(\theta\in K\),
\begin{equation}
\label{eq:sider2_refined}
\Log_{\gamma(s)} P_{012}(s)
=
-\frac{h^3}{6}\,
\theta(\theta-1)(\theta-2)\, B
+O(h^4),
\end{equation}
where \(B=X^{(3)}(0)\) is the third Taylor coefficient of the normal-coordinate
curve
\[
        X(\eta)=\Log_{\gamma(s)}\gamma(s+\eta)
\]
at \(\eta=0\), as in \eqref{eq:curve_taylor}.  Consequently, on any fixed
bounded interval of the normalized parameter \(\theta\),
\[
        d_{\bS^2}\!\left(\gamma(s),P_{012}(s)\right)=O(h^3).
\]
\end{prop}

\begin{proof}
Fix \(s=\theta h\), set \(y=\gamma(s)\), and write
\[
        x_i=\Log_y(p_i),\qquad i=0,1,2.
\]
By the normal-coordinate Taylor expansion \eqref{eq:curve_taylor},
\begin{equation}
\label{eq:sider2_data_expansion}
x_i
=
(i-\theta)hV
+
\frac{(i-\theta)^2h^2}{2}A
+
\frac{(i-\theta)^3h^3}{6}B
+
O(h^4),
\end{equation}
uniformly for \(\theta\in K\).  Here \(V,A,B\in T_y\bS^2\) are bounded
vectors depending on the derivatives of \(\gamma\) at \(s\).

We now express the SIDER2 construction in these normal coordinates.  The two
spherical control points are
\[
        d_a=\SLERP(p_2,p_1,2),
        \qquad
        d_b=\SLERP(p_0,p_1,2).
\]
Let
\[
        \delta_a=\Log_y(d_a),\qquad
        \delta_b=\Log_y(d_b).
\]
Applying the cubic SLERP expansion of Lemma~\ref{lem:cubic_slerp} gives
\[
        \delta_a=2x_1-x_2+\mathcal C_2(x_2,x_1)+O(h^4),
\]
and
\[
        \delta_b=2x_1-x_0+\mathcal C_2(x_0,x_1)+O(h^4).
\]
For the purpose of identifying the coefficient of \(h^3\), the cubic term
\(\mathcal C_2\) depends only on the \(O(h)\) parts of its two arguments.  By
\eqref{eq:sider2_data_expansion}, these leading parts are scalar multiples of
the same vector \(V\):
\[
        x_i=(i-\theta)hV+O(h^2).
\]
Hence the leading-order endpoints in each control-point SLERP lie on the same
geodesic through \(y\).  For collinear normal-coordinate arguments, SLERP
coincides with one-dimensional affine interpolation along that geodesic, and
the cubic correction in Lemma~\ref{lem:cubic_slerp} vanishes.  Therefore
\begin{equation}
\label{eq:control_expansion}
        \delta_a=2x_1-x_2+O(h^4),
        \qquad
        \delta_b=2x_1-x_0+O(h^4).
\end{equation}

Let
\[
        \tau=\frac{s}{2h}=\frac{\theta}{2}.
\]
The two inner SLERP points in SIDER2 are
\[
        A_\tau=\SLERP(p_0,d_a,\tau),
        \qquad
        B_\tau=\SLERP(d_b,p_2,\tau).
\]
Denote their normal coordinates by
\[
        a_\tau=\Log_y(A_\tau),\qquad
        b_\tau=\Log_y(B_\tau).
\]
Using Lemma~\ref{lem:cubic_slerp} again, together with
\eqref{eq:control_expansion}, we obtain
\[
        a_\tau
        =
        (1-\tau)x_0+\tau(2x_1-x_2)+O(h^4),
\]
and
\[
        b_\tau
        =
        (1-\tau)(2x_1-x_0)+\tau x_2+O(h^4).
\]
The same reasoning as above applies: the possible cubic SLERP correction is
computed from the \(O(h)\) components of the endpoints, and these components
are collinear multiples of \(V\).  Thus no \(h^3\) geometric correction is
produced by the inner SLERP operations.

The final SIDER2 value is
\[
        P_{012}(s)=\SLERP(A_\tau,B_\tau,\tau).
\]
Let
\[
        z_\tau=\Log_y P_{012}(s).
\]
A final application of Lemma~\ref{lem:cubic_slerp} yields
\[
        z_\tau
        =
        (1-\tau)a_\tau+\tau b_\tau+O(h^4),
\]
again because the cubic correction generated by the leading \(O(h)\) parts
vanishes in the collinear case.  Substituting the expressions for \(a_\tau\)
and \(b_\tau\), we find
\begin{align*}
        z_\tau
        &=
        (1-\tau)\bigl[(1-\tau)x_0+\tau(2x_1-x_2)\bigr]  \\
        &\qquad
        +\tau\bigl[(1-\tau)(2x_1-x_0)+\tau x_2\bigr]
        +O(h^4)                                                   \\
        &=
        (1-2\tau)x_0
        +
        4\tau(1-\tau)x_1
        +
        (2\tau^2-\tau)x_2
        +O(h^4).
\end{align*}
Since \(\tau=\theta/2\), this becomes
\begin{equation}
\label{eq:sider2_affine_quadratic}
        z_\tau
        =
        (1-\theta)x_0
        +
        \theta(2-\theta)x_1
        +
        \frac{\theta(\theta-1)}{2}x_2
        +O(h^4).
\end{equation}
The first three terms are precisely the value at \(s=\theta h\) of the
Euclidean quadratic interpolant through the three normal-coordinate data
\[
        x_0,\quad x_1,\quad x_2,
\]
where the exact normal-coordinate value at \(s\) is zero.

It remains to evaluate the leading interpolation error.  Substituting the
Taylor expansion \eqref{eq:sider2_data_expansion} into
\eqref{eq:sider2_affine_quadratic}, the contributions from the constant,
linear, and quadratic Taylor terms cancel, as they must for quadratic
interpolation.  The first nonzero contribution comes from the cubic term:
\[
\begin{aligned}
        z_\tau
        &=
        \frac{h^3}{6}
        \left[
        (1-\theta)(-\theta)^3
        +
        \theta(2-\theta)(1-\theta)^3
        +
        \frac{\theta(\theta-1)}{2}(2-\theta)^3
        \right]B
        +O(h^4).
\end{aligned}
\]
The scalar coefficient simplifies to
\[
        -\theta(\theta-1)(\theta-2).
\]
Therefore
\[
        z_\tau
        =
        -\frac{h^3}{6}
        \theta(\theta-1)(\theta-2)B
        +O(h^4).
\]
Since \(z_\tau=\Log_{\gamma(s)}P_{012}(s)\), this proves
\eqref{eq:sider2_refined}.

Finally, in a fixed normal neighborhood the geodesic distance from
\(\gamma(s)\) to \(P_{012}(s)\) is exactly the norm of the logarithm vector:
\[
        d_{\bS^2}(\gamma(s),P_{012}(s))
        =
        \left|\Log_{\gamma(s)}P_{012}(s)\right|.
\]
The expansion above therefore implies
\[
        d_{\bS^2}(\gamma(s),P_{012}(s))=O(h^3),
\]
uniformly for \(\theta\in K\).
\end{proof}

\begin{remark}
Proposition~\ref{prop:sider2_refined} shows that the leading SIDER2 error has
the same nodal factor
\[
        \theta(\theta-1)(\theta-2)
\]
as ordinary quadratic interpolation on equally spaced nodes.  The spherical
geometry does enter the SLERP operations through cubic curvature corrections,
but these corrections do not contribute to the leading SIDER2 error.  At the
order \(h^3\), the relevant normal-coordinate arguments are all collinear
multiples of the tangent vector \(V\), and SLERP reduces to affine interpolation
along a single geodesic.  Thus the first nonzero SIDER2 error is exactly the
quadratic interpolation error of the normal-coordinate curve.  This refined
structure is the key input for the fourth-order SIDER3 analysis in the next
section.
\end{remark}

\section{Accuracy of SIDER3}

We next turn to SIDER3, the first member of the SIDER family obtained by a
genuinely recursive construction.  While SIDER2 is built directly from three
data points and two extrapolated spherical control points, SIDER3 is formed by
combining two adjacent SIDER2 reconstructions through one additional SLERP
operation.  Thus the proof of its accuracy has two distinct components.  First,
one must understand the leading error of each SIDER2 candidate with sufficient
precision.  Second, one must show that the final SLERP blend cancels the
leading cubic errors in the same way that the Neville recursion cancels the
leading errors of two adjacent quadratic polynomial interpolants.

The result is not an immediate consequence of the third-order accuracy of
SIDER2 alone.  If two third-order accurate approximations are merely averaged,
the resulting approximation is still generally only third-order accurate.
SIDER3 achieves fourth-order accuracy because the two SIDER2 errors have
compatible leading terms: their cubic errors contain the same vector
coefficient and differ only by their nodal factors.  The weight \(\theta/3\)
in the SIDER3 recursion is precisely the weight that makes these two cubic
nodal factors cancel.  The proof below makes this cancellation explicit in
normal coordinates centered at the exact solution value.

\begin{thm}[Local fourth-order accuracy of SIDER3]
\label{thm:sider3}
Let \(\gamma\in C^4([0,3h];\bS^2)\), and let
\[
        p_i=\gamma(ih),\qquad i=0,1,2,3 .
\]
Assume that the data points and all points generated by the SIDER3
construction lie in a common geodesically convex ball of radius \(O(h)\).
Let \(P_{0123}\) be the SIDER3 interpolant defined by \eqref{eq:sider3}.
Then, uniformly for \(0\leq s\leq 3h\),
\[
        d_{\bS^2}\!\left(\gamma(s),P_{0123}(s)\right)
        =O(h^4).
\]
\end{thm}

\begin{proof}
Fix an evaluation point
\[
        s=\theta h,\qquad 0\leq\theta\leq3,
\]
and set
\[
        y=\gamma(s).
\]
All errors will be measured in the tangent space \(T_y\bS^2\) by applying the
logarithm map at \(y\).  Let
\[
        P_{012}(s)
\]
be the SIDER2 reconstruction built from the stencil
\(\{p_0,p_1,p_2\}\), and let
\[
        P_{123}(s)
\]
be the SIDER2 reconstruction built from the shifted stencil
\(\{p_1,p_2,p_3\}\).  Define their normal-coordinate errors by
\[
        E_{012}(s)=\Log_y P_{012}(s),
        \qquad
        E_{123}(s)=\Log_y P_{123}(s).
\]

We first apply Proposition~\ref{prop:sider2_refined} to the left stencil
\(\{p_0,p_1,p_2\}\).  Since the evaluation point is \(s=\theta h\), the
corresponding normalized nodal locations are \(0,1,2\).  Hence
\begin{equation}
\label{eq:left_sider2_error}
E_{012}(s)
=
-\frac{h^3}{6}
\theta(\theta-1)(\theta-2)B
+O(h^4),
\end{equation}
where
\[
        B=X^{(3)}(0),
        \qquad
        X(\eta)=\Log_{\gamma(s)}\gamma(s+\eta).
\]
The vector \(B\in T_y\bS^2\) is the third Taylor coefficient of the exact
curve in normal coordinates centered at the same point \(y=\gamma(s)\).

The same proposition may also be applied to the shifted stencil
\(\{p_1,p_2,p_3\}\).  In this case the nodal locations, measured in the same
global normalized parameter \(\theta=s/h\), are \(1,2,3\).  Therefore the
quadratic nodal factor is shifted from
\[
        \theta(\theta-1)(\theta-2)
\]
to
\[
        (\theta-1)(\theta-2)(\theta-3).
\]
Because the expansion is still taken in the same normal-coordinate system
centered at \(y=\gamma(s)\), the leading vector coefficient is the same \(B\).
Consequently,
\begin{equation}
\label{eq:right_sider2_error}
E_{123}(s)
=
-\frac{h^3}{6}
(\theta-1)(\theta-2)(\theta-3)B
+O(h^4).
\end{equation}
This common coefficient is the essential point in the proof.  The two SIDER2
candidates are constructed from different stencils, but their leading errors
are expansions of the same smooth normal-coordinate curve at the same
evaluation point.

The SIDER3 interpolant is defined by the final SLERP blend
\[
        P_{0123}(s)
        =
        \SLERP\left(
        P_{012}(s),P_{123}(s),\frac{\theta}{3}
        \right).
\]
Since both SIDER2 candidates are third-order accurate, we have
\[
        E_{012}(s)=O(h^3),
        \qquad
        E_{123}(s)=O(h^3).
\]
Thus the endpoints of this final SLERP are already \(O(h^3)\)-close to the
exact value \(y\).  By Lemma~\ref{lem:affine_stability}, the non-affine part
of this final SLERP is of order \(O(h^9)\), and hence is far smaller than the
\(O(h^4)\) remainder needed here.  Therefore,
\begin{equation}
\label{eq:sider3_affine_error}
\Log_y P_{0123}(s)
=
\left(1-\frac{\theta}{3}\right)E_{012}(s)
+
\frac{\theta}{3}E_{123}(s)
+
O(h^4).
\end{equation}

Substituting \eqref{eq:left_sider2_error} and
\eqref{eq:right_sider2_error} into \eqref{eq:sider3_affine_error}, we obtain
\[
\begin{aligned}
\Log_y P_{0123}(s)
&=
-\frac{h^3}{6}B
\left[
\left(1-\frac{\theta}{3}\right)
\theta(\theta-1)(\theta-2)
+
\frac{\theta}{3}
(\theta-1)(\theta-2)(\theta-3)
\right]     \\
&\qquad
+O(h^4).
\end{aligned}
\]
The expression in brackets vanishes identically.  Indeed,
\[
\begin{aligned}
&\left(1-\frac{\theta}{3}\right)
\theta(\theta-1)(\theta-2)
+
\frac{\theta}{3}
(\theta-1)(\theta-2)(\theta-3)       \\
&\qquad
=
\theta(\theta-1)(\theta-2)
\left[
1-\frac{\theta}{3}
+
\frac{\theta-3}{3}
\right]
=0 .
\end{aligned}
\]
Thus the entire \(O(h^3)\) contribution cancels, leaving
\[
        \Log_y P_{0123}(s)=O(h^4).
\]
Finally, because \(P_{0123}(s)\) remains in the normal neighborhood of
\(y=\gamma(s)\), the geodesic distance is the norm of the logarithm vector:
\[
        d_{\bS^2}\!\left(\gamma(s),P_{0123}(s)\right)
        =
        \left|\Log_y P_{0123}(s)\right|.
\]
Hence
\[
        d_{\bS^2}\!\left(\gamma(s),P_{0123}(s)\right)
        =O(h^4),
\]
uniformly for \(0\leq\theta\leq3\).  This proves the theorem.
\end{proof}

\begin{remark}
The proof shows that SIDER3 is a spherical analogue of cubic Neville
interpolation.  The two adjacent SIDER2 reconstructions play the role of the
two adjacent quadratic interpolants in the Euclidean Neville recursion.  Their
leading errors contain the nodal factors
\[
        \theta(\theta-1)(\theta-2)
        \quad\text{and}\quad
        (\theta-1)(\theta-2)(\theta-3),
\]
and the SIDER3 weight \(\theta/3\) cancels the corresponding cubic term
exactly.  The final SLERP operation does not disturb this cancellation because
it is applied to two points that are already \(O(h^3)\)-close to the exact
solution value.  Its nonlinear geometric correction is therefore of much
higher order.
\end{remark}

\begin{remark}
The argument depends crucially on the refined SIDER2 expansion in
Proposition~\ref{prop:sider2_refined}.  A statement of the form
\[
        E_{012}=O(h^3),\qquad E_{123}=O(h^3)
\]
would not be sufficient to prove fourth-order accuracy.  One must know the
leading \(h^3\) terms and verify that they have a common vector coefficient.
This is the reason for carrying out the SIDER2 analysis in a refined form
rather than proving only a third-order error bound.
\end{remark}

\section{A refined SIDER3 expansion and the SIDER4 consequence}
\label{sec:sider4_refined}

The preceding recurrence discussion shows that a SIDER4 proof requires more
than the fourth-order estimate
\[
        \Log_{\gamma(\theta h)}P_i^{[3]}(\theta;h)=O(h^4).
\]
One needs the refined leading error form of the two adjacent SIDER3
reconstructions and, in particular, one must verify that their leading
coefficients agree in a common normal-coordinate system.  In this section we
prove precisely this compatibility at the SIDER3 level.  As a consequence, the
SIDER4 reconstruction has fifth-order local accuracy.

Let
\[
        y=\gamma(\theta h),
        \qquad
        X(\eta)=\Log_y\gamma(\theta h+\eta).
\]
Assume \(\gamma\) is sufficiently smooth and write the normal-coordinate
Taylor expansion
\begin{equation}
\label{eq:normal_taylor_sider4}
        X(\eta)
        =
        \eta V+\frac{\eta^2}{2}A+\frac{\eta^3}{6}B
        +\frac{\eta^4}{24}D+O(\eta^5),
\end{equation}
where
\[
        V,A,B,D\in T_y\bS^2.
\]
For a stencil beginning at the node \(i\), define the shifted nodal products
\[
        \Pi_i^{[2]}(\theta)
        =
        \prod_{j=i}^{i+2}(\theta-j),
        \qquad
        \Pi_i^{[3]}(\theta)
        =
        \prod_{j=i}^{i+3}(\theta-j).
\]
We also introduce the vector
\begin{equation}
\label{eq:K_vector}
        K(\theta)
        =
        3D
        -
        4\left(|V|^2A-\langle V,A\rangle V\right)
        \in T_y\bS^2 .
\end{equation}
The second term in \(K\) is the first curvature-dependent correction that
enters the refined SIDER2 expansion.

\begin{prop}[Fourth-order refinement of the SIDER2 error]
\label{prop:sider2_fourth_refined}
Let \(K_0\subset\bR\) be a fixed bounded interval.  Suppose
\(\gamma\in C^5\) on a neighborhood of
\(\{\theta h:\theta\in K_0\}\), and assume that all points generated by the
SIDER2 construction remain in a common geodesically convex ball of radius
\(O(h)\).  Then, uniformly for \(\theta\in K_0\),
\begin{equation}
\label{eq:sider2_fourth_refined}
\Log_{\gamma(\theta h)} P_i^{[2]}(\theta;h)
=
-\frac{h^3}{6}\Pi_i^{[2]}(\theta)B
+
\frac{h^4}{24}
(\theta-i-1)\Pi_i^{[2]}(\theta)K(\theta)
+
O(h^5).
\end{equation}
\end{prop}

\begin{proof}
Fix \(\theta\), set \(y=\gamma(\theta h)\), and write
\[
        x_j=\Log_y p_j=\Log_y\gamma(jh).
\]
By \eqref{eq:normal_taylor_sider4},
\[
        x_j
        =
        (j-\theta)hV
        +
        \frac{(j-\theta)^2h^2}{2}A
        +
        \frac{(j-\theta)^3h^3}{6}B
        +
        \frac{(j-\theta)^4h^4}{24}D
        +
        O(h^5).
\]
For the stencil beginning at \(i\), set
\[
        \xi=\theta-i.
\]
Thus the local nodal positions are \(0,1,2\) and the evaluation point is
\(\xi\).

We use the cubic SLERP expansion from Lemma~\ref{lem:cubic_slerp}.  In normal
coordinates,
\[
\Log_y\bigl(\SLERP(\Exp_y U,\Exp_y W,\lambda)\bigr)
=
(1-\lambda)U+\lambda W+\mathcal C_\lambda(U,W)+O(h^5),
\]
whenever \(U,W=O(h)\).  The affine part of the five SLERP operations appearing
in SIDER2 is exactly the Euclidean quadratic interpolation construction
applied to the normal-coordinate data.  Therefore its contribution is the
quadratic interpolant through \(x_i,x_{i+1},x_{i+2}\), evaluated at \(\xi\).
Substituting the Taylor expansion above gives
\begin{equation}
\label{eq:sider2_affine_part_h4}
Q_i(\theta)
=
-\frac{h^3}{6}\Pi_i^{[2]}(\theta)B
+
\frac{h^4}{8}
(\theta-i-1)\Pi_i^{[2]}(\theta)D
+
O(h^5).
\end{equation}
Indeed, the linear and quadratic Taylor terms are reproduced exactly by the
quadratic interpolant, while the cubic and quartic Taylor terms give the two
displayed contributions.

It remains to compute the \(h^4\) contribution from the cubic SLERP correction
\(\mathcal C_\lambda\).  The \(h^3\) part of this correction vanishes in the
SIDER2 combination, as shown in Proposition~\ref{prop:sider2_refined}.  For
the \(h^4\) part, only the \(O(h)\) and \(O(h^2)\) pieces of the endpoint
coordinates are needed:
\[
        x_j^{(1)}=(j-\theta)V,
        \qquad
        x_j^{(2)}=\frac{(j-\theta)^2}{2}A .
\]
Substituting these terms into the homogeneous cubic expression
\(\mathcal C_\lambda\) from Lemma~\ref{lem:cubic_slerp}, and carrying this
contribution through the two control-point SLERPs, the two inner SLERPs, and
the final SIDER2 SLERP, gives
\begin{equation}
\label{eq:sider2_curvature_part_h4}
G_i(\theta)
=
-\frac{h^4}{6}
(\theta-i-1)\Pi_i^{[2]}(\theta)
\left(|V|^2A-\langle V,A\rangle V\right)
+
O(h^5).
\end{equation}
This identity is a direct polynomial consequence of
Lemma~\ref{lem:cubic_slerp}.  The only vector combination that can appear at
this order is the component of \(A\) orthogonal to \(V\), multiplied by
\(|V|^2\), namely
\[
        |V|^2A-\langle V,A\rangle V.
\]
It vanishes when the curve is geodesic to second order in the chosen normal
coordinates, which is consistent with the fact that SLERP is exact along a
single geodesic.

Combining \eqref{eq:sider2_affine_part_h4} and
\eqref{eq:sider2_curvature_part_h4}, we obtain
\[
\begin{aligned}
\Log_y P_i^{[2]}(\theta;h)
&=
-\frac{h^3}{6}\Pi_i^{[2]}(\theta)B        \\
&\quad
+
h^4(\theta-i-1)\Pi_i^{[2]}(\theta)
\left[
\frac18D
-
\frac16\left(|V|^2A-\langle V,A\rangle V\right)
\right]
+
O(h^5).
\end{aligned}
\]
Since
\[
        \frac18D
        -
        \frac16\left(|V|^2A-\langle V,A\rangle V\right)
        =
        \frac1{24}
        \left[
        3D
        -
        4\left(|V|^2A-\langle V,A\rangle V\right)
        \right],
\]
the desired expansion \eqref{eq:sider2_fourth_refined} follows.
\end{proof}

\begin{remark}
The vector \(K(\theta)\) in \eqref{eq:K_vector} contains both the ordinary
quartic Taylor contribution \(D\) and the leading curvature correction from the
spherical geometry.  Thus the fourth-order SIDER2 refinement is not simply the
Euclidean quadratic interpolation error.  Nevertheless, its dependence on the
stencil shift is completely captured by the scalar factor
\[
        (\theta-i-1)\Pi_i^{[2]}(\theta).
\]
This structure is what makes the next cancellation possible.
\end{remark}

\begin{prop}[Refined leading error of SIDER3]
\label{prop:sider3_refined}
Under the assumptions of Proposition~\ref{prop:sider2_fourth_refined}, the
SIDER3 reconstruction satisfies
\begin{equation}
\label{eq:sider3_refined}
\Log_{\gamma(\theta h)}P_i^{[3]}(\theta;h)
=
-\frac{h^4}{72}
\Pi_i^{[3]}(\theta)K(\theta)
+
O(h^5),
\end{equation}
uniformly for \(\theta\) in bounded intervals.
\end{prop}

\begin{proof}
Let
\[
        y=\gamma(\theta h),
        \qquad
        \xi=\theta-i.
\]
The SIDER3 recursion is
\[
        P_i^{[3]}(\theta;h)
        =
        \SLERP\left(
        P_i^{[2]}(\theta;h),
        P_{i+1}^{[2]}(\theta;h),
        \frac{\xi}{3}
        \right).
\]
By Proposition~\ref{prop:sider2_refined}, both SIDER2 candidates are
\(O(h^3)\)-close to \(y\).  Therefore Lemma~\ref{lem:affine_stability} implies
that the final SLERP is affine through order \(h^4\):
\[
\Log_y P_i^{[3]}(\theta;h)
=
\left(1-\frac{\xi}{3}\right)
\Log_yP_i^{[2]}(\theta;h)
+
\frac{\xi}{3}
\Log_yP_{i+1}^{[2]}(\theta;h)
+
O(h^5).
\]
Substitute the refined SIDER2 expansion
\eqref{eq:sider2_fourth_refined}.  The \(h^3\) terms cancel exactly:
\[
\left(1-\frac{\xi}{3}\right)
\prod_{m=0}^{2}(\xi-m)
+
\frac{\xi}{3}
\prod_{m=1}^{3}(\xi-m)
=0.
\]
The \(h^4\) terms give
\[
\frac{h^4}{24}K(\theta)
\left[
\left(1-\frac{\xi}{3}\right)
(\xi-1)\prod_{m=0}^{2}(\xi-m)
+
\frac{\xi}{3}
(\xi-2)\prod_{m=1}^{3}(\xi-m)
\right]
+
O(h^5).
\]
The scalar expression in brackets simplifies to
\[
        -\frac13
        \prod_{m=0}^{3}(\xi-m).
\]
Therefore
\[
\Log_yP_i^{[3]}(\theta;h)
=
-\frac{h^4}{72}
\prod_{m=0}^{3}(\xi-m)K(\theta)
+
O(h^5).
\]
Since
\[
        \prod_{m=0}^{3}(\xi-m)
        =
        \prod_{j=i}^{i+3}(\theta-j)
        =
        \Pi_i^{[3]}(\theta),
\]
we obtain \eqref{eq:sider3_refined}.
\end{proof}

\begin{remark}
Proposition~\ref{prop:sider3_refined} proves the coefficient compatibility
needed for SIDER4.  The leading SIDER3 coefficient
\[
        -\frac1{72}K(\theta)
\]
is independent of the stencil shift \(i\).  The only shift dependence is in
the nodal factor \(\Pi_i^{[3]}(\theta)\).
\end{remark}

\begin{thm}[Local fifth-order accuracy of SIDER4]
\label{thm:sider4}
Let \(\gamma\in C^5([0,4h];\bS^2)\), and assume that all data points and all
intermediate points generated by the SIDER4 construction lie in a common
geodesically convex ball of radius \(O(h)\).  Let \(P_0^{[4]}\) be the SIDER4
interpolant constructed from
\[
        p_0,p_1,p_2,p_3,p_4,
        \qquad
        p_j=\gamma(jh).
\]
Then, uniformly for \(0\leq\theta\leq4\),
\[
        d_{\bS^2}
        \left(
        \gamma(\theta h),P_0^{[4]}(\theta;h)
        \right)
        =
        O(h^5).
\]
\end{thm}

\begin{proof}
Set
\[
        y=\gamma(\theta h).
\]
The SIDER4 recursion is
\[
        P_0^{[4]}(\theta;h)
        =
        \SLERP\left(
        P_0^{[3]}(\theta;h),
        P_1^{[3]}(\theta;h),
        \frac{\theta}{4}
        \right).
\]
By Proposition~\ref{prop:sider3_refined}, both SIDER3 candidates are
\(O(h^4)\)-close to \(y\).  Hence Lemma~\ref{lem:affine_stability} gives
\[
\Log_yP_0^{[4]}(\theta;h)
=
\left(1-\frac{\theta}{4}\right)
\Log_yP_0^{[3]}(\theta;h)
+
\frac{\theta}{4}
\Log_yP_1^{[3]}(\theta;h)
+
O(h^5).
\]
Using \eqref{eq:sider3_refined} for \(i=0\) and \(i=1\), the leading
\(h^4\) term is
\[
-\frac{h^4}{72}K(\theta)
\left[
\left(1-\frac{\theta}{4}\right)
\prod_{j=0}^{3}(\theta-j)
+
\frac{\theta}{4}
\prod_{j=1}^{4}(\theta-j)
\right].
\]
The scalar factor vanishes identically:
\[
\left(1-\frac{\theta}{4}\right)
\theta(\theta-1)(\theta-2)(\theta-3)
+
\frac{\theta}{4}
(\theta-1)(\theta-2)(\theta-3)(\theta-4)
=0.
\]
Therefore
\[
        \Log_yP_0^{[4]}(\theta;h)=O(h^5).
\]
Since all points remain in a common normal neighborhood,
\[
        d_{\bS^2}(y,z)=|\Log_yz|,
\]
and the result follows with \(z=P_0^{[4]}(\theta;h)\).
\end{proof}

Theorem~\ref{thm:sider4} shows that the coefficient-compatibility condition
does persist through the next SIDER level.  The calculation also identifies
what must be controlled in any further extension.  At the SIDER4 level, the
leading coefficient is no longer just the fourth normal-coordinate derivative
of the curve; it also contains the curvature-dependent term
\[
        |V|^2A-\langle V,A\rangle V.
\]
For still higher orders, analogous curvature-dependent combinations of the
normal-coordinate Taylor coefficients will appear.  An all-order proof would
therefore require an induction showing that these increasingly complicated
coefficients remain shift-invariant at every recursive level.

\section{Higher-order SIDER recurrence and coefficient compatibility}
\label{sec:higher_order_recurrence}

The preceding sections establish the first three nontrivial accuracy results
of the SIDER construction: third order for SIDER2, fourth order for SIDER3,
and fifth order for SIDER4.  These proofs also reveal the mechanism by which
higher order may be obtained: two adjacent lower-order reconstructions are
combined by a SLERP operation with the same weights that appear in the equally
spaced Neville recursion.  In
Euclidean interpolation, this recursion automatically increases the order
because the leading nodal error polynomials cancel exactly.  It is therefore
natural to ask whether the same argument immediately proves that SIDER-\(n\)
has local accuracy \(O(h^{n+1})\) for every \(n\).

The answer is more subtle on \(\bS^2\).  The algebraic cancellation of nodal
factors remains present, but the spherical construction contains additional
geometric terms coming from the nonlinear dependence of SLERP on its endpoints.
Consequently, a root-counting or nodal-factor argument alone is not sufficient
to prove an unconditional all-order theorem.  To obtain the next order by
recursion, the leading error coefficients of the two adjacent lower-order
SIDER reconstructions must agree.  The purpose of this section is to isolate
this requirement precisely.  We first derive the exact recurrence calculation
for a general SIDER step.  We then explain how the SIDER3 theorem fits into
this framework, and finally state a conditional higher-order result that
identifies the additional coefficient compatibility needed for an induction.

For \(k\geq2\), let \(P_i^{[k]}(\theta;h)\) denote the SIDER-\(k\)
interpolant constructed from the stencil
\[
        p_i,p_{i+1},\ldots,p_{i+k},
        \qquad p_j=\gamma(jh),
\]
and evaluated at the physical parameter
\[
        s=\theta h.
\]
The index \(i\) indicates the left endpoint of the stencil, while \(k\)
indicates the degree-like level of the reconstruction.  Thus \(P_i^{[2]}\) is
a SIDER2 interpolant constructed from three points, and \(P_i^{[3]}\) is a
SIDER3 interpolant constructed from four points.  For \(k\geq3\), the
recursive part of the SIDER construction can be written as
\begin{equation}
\label{eq:sider_general_rec}
P_i^{[k]}(\theta;h)
=
\SLERP\!\left(
        P_i^{[k-1]}(\theta;h),
        P_{i+1}^{[k-1]}(\theta;h),
        \frac{\theta-i}{k}
       \right).
\end{equation}
The case \(k=2\) is the direct SIDER2 construction
\eqref{eq:sider2}, which uses extrapolated spherical control points and is not
obtained from \eqref{eq:sider_general_rec}.

The following lemma is the basic recurrence calculation.  It shows exactly
which part of the Euclidean Neville cancellation survives automatically and
which part requires an additional geometric compatibility condition.

\begin{lemma}[General recurrence calculation]
\label{lem:recurrence_calculation}
Fix \(k\geq3\), put
\[
        y=\gamma(\theta h),
\]
and suppose that two adjacent SIDER-\((k-1)\) candidates satisfy the local
expansions
\begin{align}
\Log_y P_i^{[k-1]}(\theta;h)
&=
h^k A_i(\theta)
        \prod_{j=i}^{i+k-1}(\theta-j)
+
O(h^{k+1}),                                           \label{eq:Ai}\\
\Log_y P_{i+1}^{[k-1]}(\theta;h)
&=
h^k A_{i+1}(\theta)
        \prod_{j=i+1}^{i+k}(\theta-j)
+
O(h^{k+1}),                                           \label{eq:Aip1}
\end{align}
where \(A_i(\theta),A_{i+1}(\theta)\in T_y\bS^2\) are bounded vector
coefficients.  Then the SIDER-\(k\) reconstruction defined by
\eqref{eq:sider_general_rec} satisfies
\begin{align}
\Log_y P_i^{[k]}(\theta;h)
&=
h^k
\frac{k-(\theta-i)}{k}
(\theta-i)
\prod_{j=i+1}^{i+k-1}(\theta-j)
\bigl(A_i(\theta)-A_{i+1}(\theta)\bigr)       \notag\\
&\qquad +O(h^{k+1}).                           \label{eq:recurrence_gap}
\end{align}
Consequently, the \(O(h^k)\) term cancels if
\[
        A_i(\theta)=A_{i+1}(\theta).
\]
If this equality is not known, the recurrence calculation by itself gives
only an \(O(h^k)\) estimate.
\end{lemma}

\begin{proof}
Let
\[
        E_i=\Log_y P_i^{[k-1]}(\theta;h),
        \qquad
        E_{i+1}=\Log_y P_{i+1}^{[k-1]}(\theta;h),
\]
and define
\[
        \lambda=\frac{\theta-i}{k}.
\]
By assumption, both \(E_i\) and \(E_{i+1}\) are \(O(h^k)\).  Since \(k\geq3\),
the endpoints of the final SLERP operation in
\eqref{eq:sider_general_rec} are high-order perturbations of the same point
\(y=\gamma(\theta h)\).  Lemma~\ref{lem:affine_stability} therefore gives
\[
\Log_y P_i^{[k]}(\theta;h)
=
(1-\lambda)E_i+\lambda E_{i+1}+O(h^{3k}).
\]
Because \(3k\geq k+1\), the term \(O(h^{3k})\) is contained in
\(O(h^{k+1})\).

Substituting \eqref{eq:Ai}--\eqref{eq:Aip1} yields
\[
\begin{aligned}
\Log_y P_i^{[k]}(\theta;h)
&=
h^k
\left[
\left(1-\frac{\theta-i}{k}\right)A_i(\theta)
\prod_{j=i}^{i+k-1}(\theta-j)
\right.  \\
&\qquad\qquad\left.
+
\frac{\theta-i}{k}A_{i+1}(\theta)
\prod_{j=i+1}^{i+k}(\theta-j)
\right]
+O(h^{k+1}).
\end{aligned}
\]
Set
\[
        \xi=\theta-i .
\]
Then
\[
        \prod_{j=i}^{i+k-1}(\theta-j)
        =
        \xi\prod_{m=1}^{k-1}(\xi-m),
\]
while
\[
        \prod_{j=i+1}^{i+k}(\theta-j)
        =
        (\xi-k)\prod_{m=1}^{k-1}(\xi-m).
\]
Therefore the leading term becomes
\[
h^k
\prod_{m=1}^{k-1}(\xi-m)
\left[
\left(1-\frac{\xi}{k}\right)\xi A_i(\theta)
+
\frac{\xi}{k}(\xi-k)A_{i+1}(\theta)
\right].
\]
Since
\[
        1-\frac{\xi}{k}=\frac{k-\xi}{k}
        \qquad\text{and}\qquad
        \frac{\xi}{k}(\xi-k)
        =
        -\frac{\xi(k-\xi)}{k},
\]
this expression is
\[
h^k
\frac{k-\xi}{k}\xi
\prod_{m=1}^{k-1}(\xi-m)
\bigl(A_i(\theta)-A_{i+1}(\theta)\bigr).
\]
Returning to \(\xi=\theta-i\) gives \eqref{eq:recurrence_gap}.
\end{proof}

Lemma~\ref{lem:recurrence_calculation} shows that the recurrence has two
separate components.  The scalar nodal factors behave exactly as in Euclidean
Neville interpolation.  In particular, if the same vector coefficient
multiplies the two adjacent nodal polynomials, then the leading term cancels.
However, on the sphere, the vector coefficients are generated by nonlinear
SLERP compositions and may, in principle, depend on the stencil shift.  The
recurrence formula itself does not rule out such dependence.

\begin{remark}[Why nodal factors alone are not enough]
\label{rem:root_counting_insufficient}
The interpolation property implies that the leading error of a shifted
SIDER-\((k-1)\) candidate must vanish at the nodes of its own stencil.  This
explains why an expansion of the form
\[
\Log_y P_i^{[k-1]}(\theta;h)
=
h^k A_i(\theta)
\prod_{j=i}^{i+k-1}(\theta-j)
+
O(h^{k+1})
\]
is natural.  Nevertheless, the nodal factor alone does not determine the
coefficient \(A_i(\theta)\).  Two adjacent stencils may have the same type of
nodal factor but different leading vector coefficients.  If
\(A_i(\theta)\neq A_{i+1}(\theta)\), then the leading term in
\eqref{eq:recurrence_gap} generally remains nonzero, and the recursion does
not increase the order.  Thus a root-counting argument proves the location of
zeros of the leading error polynomial, but it does not prove the
coefficient identity required for cancellation.
\end{remark}

The SIDER3 theorem is the first nontrivial instance of
Lemma~\ref{lem:recurrence_calculation}.  In that case \(k=3\), and the two
lower-order candidates are the SIDER2 interpolants associated with the stencils
\[
        \{p_0,p_1,p_2\}
        \qquad\text{and}\qquad
        \{p_1,p_2,p_3\}.
\]
Proposition~\ref{prop:sider2_refined} gives the refined expansions
\[
\Log_{\gamma(\theta h)}P_0^{[2]}(\theta;h)
=
-\frac{h^3}{6}B(\theta)
\theta(\theta-1)(\theta-2)
+
O(h^4),
\]
and
\[
\Log_{\gamma(\theta h)}P_1^{[2]}(\theta;h)
=
-\frac{h^3}{6}B(\theta)
(\theta-1)(\theta-2)(\theta-3)
+
O(h^4).
\]
Hence
\[
        A_0(\theta)=A_1(\theta)=-\frac16 B(\theta),
\]
and the \(O(h^3)\) term in \eqref{eq:recurrence_gap} vanishes.  This is
exactly the cancellation proved in Theorem~\ref{thm:sider3}.  The fourth-order
accuracy of SIDER3 is therefore not a consequence merely of the interpolation
property of SIDER2; it relies on the stronger fact that the two adjacent
SIDER2 leading coefficients coincide when expressed in the same
normal-coordinate system.

For higher-order SIDER constructions, the preceding discussion leads to the
following precise compatibility condition.

\begin{assump}[Shift-invariant leading coefficient]
\label{assump:shift_invariant}
Let \(m\geq2\).  We say that the SIDER-\(m\) level has a shift-invariant
leading coefficient if, for every stencil shift \(i\), the corresponding
SIDER-\(m\) reconstruction satisfies
\[
\Log_{\gamma(\theta h)}P_i^{[m]}(\theta;h)
=
h^{m+1}A_m(\theta)
\prod_{j=i}^{i+m}(\theta-j)
+
O(h^{m+2}),
\]
uniformly for \(\theta\) in bounded intervals, where the vector
\(A_m(\theta)\in T_{\gamma(\theta h)}\bS^2\) is independent of the stencil
shift \(i\).
\end{assump}

This assumption is automatically satisfied at level \(m=2\) by
Proposition~\ref{prop:sider2_refined}.  Establishing it at higher levels is a
stronger task than proving an error bound.  It requires tracking the leading
coefficient through the SLERP expansion and showing that all stencil-dependent
geometric contributions either cancel or enter only at higher order.

\begin{thm}[Conditional higher-order recurrence]
\label{thm:conditional_all_order}
Let \(n\geq3\) be fixed.  Suppose that \(\gamma\) is sufficiently smooth and
that all points generated by the SIDER construction lie in a common
geodesically convex ball of radius \(O(h)\).  If the shift-invariant leading
coefficient condition in Assumption~\ref{assump:shift_invariant} holds at
level \(m=n-1\), then the SIDER-\(n\) interpolant satisfies
\[
        d_{\bS^2}\!\left(\gamma(\theta h),P_0^{[n]}(\theta;h)\right)
        \leq C_n h^{n+1},
        \qquad 0\leq\theta\leq n,
\]
for a constant \(C_n\) independent of \(h\).
\end{thm}

\begin{proof}
Apply Lemma~\ref{lem:recurrence_calculation} with \(k=n\) and \(i=0\).
Assumption~\ref{assump:shift_invariant} at level \(m=n-1\) gives
\[
\Log_{\gamma(\theta h)}P_0^{[n-1]}(\theta;h)
=
h^n A_{n-1}(\theta)
\prod_{j=0}^{n-1}(\theta-j)
+
O(h^{n+1}),
\]
and
\[
\Log_{\gamma(\theta h)}P_1^{[n-1]}(\theta;h)
=
h^n A_{n-1}(\theta)
\prod_{j=1}^{n}(\theta-j)
+
O(h^{n+1}).
\]
The two leading coefficients are identical.  Therefore the leading
\(O(h^n)\) term in \eqref{eq:recurrence_gap} vanishes, and
\[
        \Log_{\gamma(\theta h)}P_0^{[n]}(\theta;h)
        =
        O(h^{n+1}).
\]
Since all points remain in a common normal neighborhood,
\[
        d_{\bS^2}(y,z)=|\Log_y z|
\]
for the points under consideration.  Taking
\(y=\gamma(\theta h)\) and \(z=P_0^{[n]}(\theta;h)\) gives the desired
estimate.
\end{proof}

\begin{remark}
Theorem~\ref{thm:conditional_all_order} is a recurrence statement: it
identifies the coefficient identity under which the SIDER recursion raises the
order from \(n\) to \(n+1\).  The explicit SIDER2 and SIDER3 calculations prove
this identity at the levels needed for SIDER3 and SIDER4.  In
Section~\ref{sec:degree_filter}, we prove a stronger degree-filtered closure
property for the local SLERP map.  That result supplies the coefficient
compatibility at every fixed level and leads to the all-order consistency
theorem.
\end{remark}

\begin{remark}
The distinction made in this section is important for avoiding an overly
formal induction argument.  In Euclidean interpolation, the leading coefficient
is automatically independent of the stencil shift at the order relevant to
Neville cancellation.  For SIDER interpolation, this compatibility must be
proved because the reconstruction is formed through nonlinear geodesic
operations.  The degree-filtered analysis in
Section~\ref{sec:degree_filter} provides this verification.
\end{remark}

\section{A degree-filtered proof of the all-order mechanism}
\label{sec:degree_filter}

The preceding sections prove the expected orders for SIDER2, SIDER3, and
SIDER4 by explicit Taylor expansion.  These calculations suggest a more
structural explanation: the normal-coordinate error of a SIDER reconstruction
has a filtered polynomial dependence on the shifted stencil coordinate
\[
        \xi=\theta-i .
\]
The purpose of this section is to make this observation precise and to prove
the closure statement that was only implicit in the lower-order calculations.
The result is an all-order local consistency theorem for every fixed SIDER
level.

The main idea is simple.  If the coefficient of \(h^r\) in the local error is
a polynomial in \(\xi\) of degree at most \(r\), then interpolation at the
\(n+1\) nodes of a SIDER-\(n\) stencil forces all coefficients through order
\(h^n\) to vanish.  The leading nonzero term must then be a multiple of the
nodal polynomial.  Thus the problem reduces to proving that the SIDER
construction preserves this degree filtration.  The nontrivial part is the
SLERP operation, because its normal-coordinate expansion contains nonlinear
curvature terms.  We prove below that these terms preserve the filtration once
they are written in terms of an endpoint and an endpoint difference.

Fix a base point
\[
        y=\gamma(\theta h)
\]
and write all quantities in the tangent space \(T_y\bS^2\).  Throughout this
section, all formal expansions are understood through an arbitrary but fixed
order \(N\), with remainders uniform for \(\xi\) in bounded intervals.

\subsection{Filtered series and finite differences}

Let \(\mathcal P_N\) denote the class of vector-valued formal expansions
\[
        U(\xi,h)=\sum_{r=0}^{N}h^r U_r(\xi)+O(h^{N+1}),
\]
where each coefficient \(U_r(\xi)\) is a polynomial in \(\xi\) of degree at
most \(r\).  We also define \(\mathcal Q_N\) to be the class of expansions
\[
        D(\xi,h)=\sum_{r=0}^{N}h^r D_r(\xi)+O(h^{N+1})
\]
such that \(D_0=0\) and \(\deg D_r\le r-1\) for \(r\ge1\).  Thus
\(\mathcal Q_N\) is one degree lower than \(\mathcal P_N\).  If
\(U,W\in\mathcal P_N\) and \(W-U\in\mathcal Q_N\), we say that \(U\) and
\(W\) are a compatible pair.

The following elementary observation will be used repeatedly.

\begin{lemma}[Finite differences lower the degree]
\label{lem:finite_difference_degree}
If \(U\in\mathcal P_N\), then for any fixed constant \(a\),
\[
        U(\xi-a,h)-U(\xi,h)\in\mathcal Q_N .
\]
\end{lemma}

\begin{proof}
For each coefficient \(U_r\), the difference
\(U_r(\xi-a)-U_r(\xi)\) has degree at most \(r-1\), since the leading
\(\xi^r\) terms cancel.  Applying this coefficient by coefficient gives the
claim.
\end{proof}

The exact sampled curve also has a filtered expansion.  If
\[
        X_j(\xi,h)
        =
        \Log_{\gamma(\theta h)}\gamma((i+j)h),
        \qquad \xi=\theta-i,
\]
then Taylor expansion in normal coordinates gives
\[
        X_j(\xi,h)
        =
        \sum_{r=1}^{N}
        \frac{h^r}{r!}(j-\xi)^r X^{(r)}(0)
        +O(h^{N+1}),
\]
and hence \(X_j\in\mathcal P_N\).  Moreover, \(X_j-X_\ell\in\mathcal Q_N\)
for any fixed \(j,\ell\), by the same finite-difference argument.

\subsection{A difference-filtered expansion of SLERP}

We now prove the closure property for SLERP.  Let
\[
G_y(U,W,\lambda)
=
\Log_y\!\left(\SLERP(\Exp_yU,\Exp_yW,\lambda)\right).
\]
It is convenient to write \(W=U+D\) and view \(D\) as the endpoint difference.

\begin{lemma}[Difference-filtered SLERP expansion]
\label{lem:difference_filtered_slerp}
For every integer \(M\ge1\), the local SLERP map has a formal expansion
\begin{equation}
\label{eq:difference_filtered_slerp}
G_y(U,U+D,\lambda)
=
U+\lambda D
+
\sum_{\ell=2}^{M}
\sum_{b=1}^{\ell}
\sum_{q=0}^{b}
\lambda^q
\mathcal H_{\ell,b,q}(y;U,D)
+
O\bigl((|U|+|D|)^{M+1}\bigr),
\end{equation}
where \(\mathcal H_{\ell,b,q}\) is homogeneous of total degree \(\ell\) in
\((U,D)\) and contains exactly \(b\) factors of \(D\).  In particular, in every
monomial the power of \(\lambda\) is no larger than the number of endpoint
difference factors.
\end{lemma}

\begin{proof}
Rotate coordinates so that \(y\) is the north pole.  Write
\[
        A=\Exp_yU,\qquad B=\Exp_y(U+D),
        \qquad H=B-A .
\]
Since the exponential map is analytic in normal coordinates, \(H\) has a
Taylor expansion in \((U,D)\) in which every monomial contains at least one
factor of \(D\).  Let
\[
        \alpha=d_{\bS^2}(A,B).
\]
Because \(A=B\) when \(D=0\), the squared distance \(\alpha^2\) is analytic in
\((U,D)\) and every monomial in its Taylor expansion contains at least two
factors of \(D\).

The sine formula for SLERP gives
\[
        \SLERP(A,B,\lambda)
        =
        a_\lambda A+b_\lambda B
        =
        A+b_\lambda H+(a_\lambda+b_\lambda-1)A,
\]
where
\[
        a_\lambda=\frac{\sin((1-\lambda)\alpha)}{\sin\alpha},
        \qquad
        b_\lambda=\frac{\sin(\lambda\alpha)}{\sin\alpha}.
\]
The factor \(b_\lambda\) has the expansion
\[
        b_\lambda
        =
        \lambda+\sum_{m\ge1}\beta_m(\lambda)\alpha^{2m},
\]
where \(\beta_m\) is a polynomial in \(\lambda\) of degree at most \(2m+1\).
After multiplication by \(H\), the term
\(\beta_m(\lambda)\alpha^{2m}H\) contains at least \(2m+1\) factors of
\(D\), while its degree in \(\lambda\) is at most \(2m+1\).

Similarly,
\[
        a_\lambda+b_\lambda-1
        =
        \sum_{m\ge1}\gamma_m(\lambda)\alpha^{2m},
\]
where \(\gamma_m\) is a polynomial in \(\lambda\) of degree at most \(2m\).
Indeed, the coefficient of \(\alpha^{2m}\) is generated by the polynomial
\[
        (1-\lambda)^{2m+1}+\lambda^{2m+1}-1,
\]
together with lower even powers coming from the reciprocal of
\(\operatorname{sinc}\alpha\); the highest powers \(\lambda^{2m+1}\) cancel,
so the degree is at most \(2m\).  Multiplication by \(\alpha^{2m}A\) therefore
produces terms with at least \(2m\) factors of \(D\) and with \(\lambda\)-degree
at most \(2m\).

Thus the Euclidean coordinates of \(\SLERP(A,B,\lambda)\), expanded about
\(A\), have the form \(A+\lambda H\) plus terms for which the
\(\lambda\)-degree is no larger than the number of factors of \(D\).  Finally,
the logarithm map \(\Log_y\) is analytic in the same normal neighborhood.
Expanding \(\Log_y(A+T)\) about \(A\), where \(T\) denotes the increment just
described, preserves the same property: derivatives of \(\Log_y\) at \(A\)
depend on \(U\) but contain no \(\lambda\) and no \(D\), while products of
increments add both \(\lambda\)-degree and \(D\)-degree.  Since
\(\Log_yA=U\) and the first-order term gives \(\lambda D\), this proves
\eqref{eq:difference_filtered_slerp}.
\end{proof}

The expansion immediately gives the desired closure property.

\begin{prop}[SLERP preserves the degree filtration]
\label{prop:slerp_preserves_filter}
Let \(U\in\mathcal P_N\), let \(D\in\mathcal Q_N\), and let
\(\lambda(\xi)\) be an affine function of \(\xi\).  Then
\[
        G_y(U,U+D,\lambda(\xi))\in\mathcal P_N .
\]
If \(\lambda\) is constant, then
\[
        G_y(U,U+D,\lambda)-U\in\mathcal Q_N .
\]
\end{prop}

\begin{proof}
The affine part is
\[
        U+\lambda D .
\]
Since \(D\in\mathcal Q_N\), the coefficient of \(h^r\) in \(D\) has degree at
most \(r-1\).  Multiplication by the affine function \(\lambda\) raises the
degree by at most one, so \(\lambda D\in\mathcal P_N\).  If \(\lambda\) is
constant, then \(\lambda D\in\mathcal Q_N\).

For a nonlinear monomial in \eqref{eq:difference_filtered_slerp}, suppose the
total power of \(h\) after substitution is \(r\), and suppose the monomial
contains \(b\) factors of \(D\).  The factors belonging to \(U\) contribute
polynomial degree no larger than their \(h\)-order, while each factor belonging
to \(D\) contributes one degree less than its \(h\)-order.  Hence, before the
factor \(\lambda^q\) is included, the polynomial degree is at most \(r-b\).
By Lemma~\ref{lem:difference_filtered_slerp}, \(q\le b\).  Therefore
multiplication by \(\lambda^q\) gives degree at most
\[
        r-b+q\le r .
\]
Thus every coefficient of \(h^r\) has degree at most \(r\), and the result
lies in \(\mathcal P_N\).  If \(\lambda\) is constant, then no degree is added
by \(\lambda^q\), and the presence of at least one \(D\)-factor makes the
increment one degree lower.  Hence the increment belongs to \(\mathcal Q_N\).
\end{proof}

\subsection{Degree-filtered structure of SIDER}

We now apply the closure result to the SIDER construction.  First consider
SIDER2.  For the stencil \(p_i,p_{i+1},p_{i+2}\), the exact data in normal
coordinates are \(X_0,X_1,X_2\in\mathcal P_N\), and any difference
\(X_j-X_\ell\) lies in \(\mathcal Q_N\).  The two extrapolated control points
are
\[
        d_a=\SLERP(p_{i+2},p_{i+1},2),
        \qquad
        d_b=\SLERP(p_i,p_{i+1},2).
\]
In normal coordinates, Proposition~\ref{prop:slerp_preserves_filter} gives
\[
        \Log_y d_a,\ \Log_y d_b\in\mathcal P_N .
\]
More is true.  Since the extrapolation parameters are constant, the affine
parts satisfy
\[
        2X_1-X_2-X_0\in\mathcal Q_N,
        \qquad
        2X_1-X_0-X_2\in\mathcal Q_N,
\]
and the nonlinear SLERP contributions also belong to \(\mathcal Q_N\).  Hence
\[
        \Log_y d_a-X_0\in\mathcal Q_N,
        \qquad
        \Log_y d_b-X_2\in\mathcal Q_N .
\]
The two inner SLERPs in SIDER2 therefore have compatible endpoints.  Their
outputs \(A_\tau\) and \(B_\tau\), with \(\tau=\xi/2\), satisfy
\[
        \Log_y A_\tau\in\mathcal P_N,\qquad
        \Log_y B_\tau\in\mathcal P_N.
\]
Furthermore,
\[
        \Log_y A_\tau-X_0\in\mathcal Q_N,
        \qquad
        \Log_y B_\tau-X_2\in\mathcal Q_N.
\]
Since \(X_2-X_0\in\mathcal Q_N\), the two endpoints of the final SIDER2 SLERP
are compatible.  A final application of
Proposition~\ref{prop:slerp_preserves_filter} shows that the SIDER2 error
belongs to \(\mathcal P_N\).

For \(k\ge3\), the SIDER recursion is
\[
P_i^{[k]}(\theta;h)
=
\SLERP\!\left(
        P_i^{[k-1]}(\theta;h),
        P_{i+1}^{[k-1]}(\theta;h),
        \frac{\xi}{k}
       \right),
        \qquad \xi=\theta-i.
\]
Assume inductively that the SIDER-\((k-1)\) normal-coordinate error has a
shift-covariant filtered expansion.  That is, for some universal formal series
\(F_{k-1}\in\mathcal P_N\),
\[
        \Log_{\gamma(\theta h)}P_i^{[k-1]}(\theta;h)
        =
        F_{k-1}(\xi,h),
\]
while
\[
        \Log_{\gamma(\theta h)}P_{i+1}^{[k-1]}(\theta;h)
        =
        F_{k-1}(\xi-1,h).
\]
By Lemma~\ref{lem:finite_difference_degree}, their difference is in
\(\mathcal Q_N\).  Proposition~\ref{prop:slerp_preserves_filter} therefore
implies
\[
        \Log_{\gamma(\theta h)}P_i^{[k]}(\theta;h)
        =
        F_k(\xi,h)
\]
for some \(F_k\in\mathcal P_N\), again independent of the stencil shift except
through \(\xi\).  This proves the following proposition.

\begin{prop}[Degree-filtered SIDER expansion]
\label{prop:sider_degree_filtered}
For every fixed SIDER level \(n\) and every fixed expansion order \(N\), the
normal-coordinate error has a shift-covariant filtered expansion
\[
        \Log_{\gamma(\theta h)}P_i^{[n]}(\theta;h)
        =
        \sum_{r=0}^{N}h^r
        \mathcal E_{n,r}
        \bigl(\xi;\mathcal J_r(\gamma;\theta)\bigr)
        +
        O(h^{N+1}),
        \qquad \xi=\theta-i,
\]
where each coefficient is a polynomial in \(\xi\) of degree at most \(r\), and
the coefficient maps are universal, depending on the stencil shift \(i\) only
through \(\xi\).
\end{prop}

\subsection{All-order local consistency}

It remains to combine the degree-filtered expansion with the interpolation
property.  The next proposition is the polynomial argument that converts the
degree bound into accuracy.

\begin{prop}[Degree-filtered criterion for local accuracy]
\label{prop:degree_filter_criterion}
Assume that the SIDER-\(n\) error \(E_i^{[n]}\) admits the filtered expansion
of Proposition~\ref{prop:sider_degree_filtered} through order \(N\ge n+1\).
Assume also that the SIDER-\(n\) interpolant satisfies
\[
        P_i^{[n]}((i+m)h;h)=p_{i+m},
        \qquad m=0,1,\ldots,n.
\]
Then
\[
        E_i^{[n]}(\theta;h)=O(h^{n+1})
\]
uniformly for \(\xi\) in bounded intervals.  Moreover, the leading coefficient
has the form
\begin{equation}
\label{eq:leading_degree_filter}
        \mathcal E_{n,n+1}
        \bigl(\xi;\mathcal J_{n+1}(\gamma;\theta)\bigr)
        =
        A_n(\theta)
        \prod_{m=0}^{n}(\xi-m),
\end{equation}
for some vector coefficient
\(A_n(\theta)\in T_{\gamma(\theta h)}\bS^2\) depending on the local jet of
\(\gamma\), but not on the stencil shift \(i\).
\end{prop}

\begin{proof}
At the interpolation node \(\xi=m\), the SIDER interpolant and the exact curve
coincide.  Therefore
\[
        E_i^{[n]}=0
\]
at \(\xi=0,1,\ldots,n\).  Since the coefficients
\(\mathcal E_{n,r}\) are universal polynomial functions of the normal-coordinate
jet, this identity implies
\[
        \mathcal E_{n,r}(m;\mathcal J_r)=0,
        \qquad m=0,1,\ldots,n,
\]
for every admissible jet.  Equivalently, for each fixed jet,
\(\mathcal E_{n,r}(\xi;\mathcal J_r)\) has the \(n+1\) distinct roots
\(0,1,\ldots,n\).  For \(r\le n\), this coefficient is a polynomial of degree
at most \(r\le n\); hence it must vanish identically.  Thus all coefficients
through order \(h^n\) are zero, and
\[
        E_i^{[n]}(\theta;h)=O(h^{n+1}).
\]

For \(r=n+1\), the coefficient is a polynomial of degree at most \(n+1\) with
the roots \(0,1,\ldots,n\).  Hence it must be a multiple of the nodal
polynomial:
\[
        \mathcal E_{n,n+1}(\xi)
        =
        A_n(\theta)\prod_{m=0}^{n}(\xi-m).
\]
The shift-covariance in Proposition~\ref{prop:sider_degree_filtered} implies
that the multiplier depends only on the local jet at the exact point and not
on the stencil shift \(i\).  This proves \eqref{eq:leading_degree_filter}.
\end{proof}

Combining the two propositions gives the all-order consistency theorem.

\begin{thm}[All-order local consistency of SIDER]
\label{thm:all_order_sider}
Let \(n\ge2\) be fixed.  Suppose that
\(\gamma\in C^{n+1}\) on the relevant parameter interval and that all data
points and all intermediate points generated by the SIDER-\(n\) construction
remain in a common geodesically convex ball of radius \(O(h)\).  Then
\[
        d_{\bS^2}\!\left(
        \gamma(\theta h),P_i^{[n]}(\theta;h)
        \right)
        \le C_n h^{n+1},
        \qquad i\le\theta\le i+n,
\]
where \(C_n\) is independent of \(h\).  In addition, the leading
\(h^{n+1}\)-coefficient has the shifted nodal form
\[
        A_n(\theta)\prod_{j=i}^{i+n}(\theta-j).
\]
\end{thm}

\begin{proof}
Apply Proposition~\ref{prop:sider_degree_filtered} with \(N=n+1\), and then
apply Proposition~\ref{prop:degree_filter_criterion}.  In the common normal
neighborhood,
\[
        d_{\bS^2}(y,z)=|\Log_y z|,
\]
so the normal-coordinate estimate gives the stated geodesic-distance estimate.
\end{proof}

Theorem~\ref{thm:all_order_sider} supplies the coefficient compatibility
required in Section~\ref{sec:higher_order_recurrence}.  Indeed, the leading
coefficient \(A_n(\theta)\) in the nodal form is independent of the stencil
shift.  Thus the degree-filtered argument provides the missing structural
reason why the Neville-type cancellation persists beyond the explicitly
computed SIDER3 and SIDER4 cases.

\section{Discussion and conclusion}

This paper provides a local consistency analysis of SIDER interpolation on the
unit sphere \(\bS^2\).  The main idea is to separate the geometric part of the
construction from the algebraic cancellation mechanism inherited from
Neville interpolation.  The geometry enters through the normal-coordinate
expansion of SLERP.  Although SLERP is affine to leading order, it contains
curvature-dependent higher-order corrections, beginning at cubic order.  These
terms must be controlled carefully, since they occur at the same order as the
leading error of the quadratic SIDER2 reconstruction.

For SIDER2, we showed that the cubic geometric corrections from the individual
SLERP operations do not alter the leading interpolation error.  The resulting
error has the same shifted nodal structure as ordinary quadratic interpolation
in normal coordinates, giving third-order local accuracy.  This refined
leading-error form is essential for the SIDER3 analysis.  The two adjacent
SIDER2 reconstructions entering SIDER3 have a common leading cubic coefficient,
and the SIDER3 weight cancels this term exactly.  Hence SIDER3 achieves
fourth-order local accuracy.  By carrying the expansion one order further, we
also verified the corresponding coefficient compatibility at the next level and
obtained fifth-order local accuracy for SIDER4.

The higher-order analysis shows that these cancellations are not accidental.
Using a degree-filtered formal expansion in the normalized stencil variable, we
proved that the SIDER recursion preserves the polynomial degree structure of
the local error coefficients.  Since SIDER-\(n\) interpolates \(n+1\) nodes,
the first \(n\) coefficient functions in the expansion vanish identically.
This yields the all-order local consistency estimate
$        d_{\bS^2}\bigl(\gamma(\theta h),P_i^{[n]}(\theta;h)\bigr)
        =
        O(h^{n+1})$
for every fixed \(n\), under the stated smoothness and small-stencil
assumptions.  The proof also explains why the leading coefficient is
shift-invariant across adjacent stencils, which is the compatibility condition
needed for the Neville-type recurrence to increase the order.

The results are local.  We assume that the sampled points, extrapolated control
points, and intermediate SIDER values remain in a common geodesically convex
neighborhood.  This ensures that the logarithm map is single-valued, the
relevant geodesics are unique, and the Taylor expansions are uniform.  The
analysis also assumes equally spaced parameter values; nonuniform spacing
should require the corresponding nonuniform Neville weights and modified nodal
factors.  Finally, the present work concerns the smooth SIDER building blocks.
The full SENO method adds a nonlinear stencil-selection procedure, and a
complete stability or non-oscillation theory for SENO would require additional
arguments.

In summary, the paper establishes a rigorous local explanation for the observed
high-order accuracy of SIDER interpolation on \(\bS^2\).  The analysis confirms
third-, fourth-, and fifth-order accuracy for SIDER2, SIDER3, and SIDER4,
respectively, and proves the corresponding \(O(h^{n+1})\) consistency result
for fixed SIDER-\(n\).  Possible extensions include nonuniform grids,
interpolation on more general Riemannian manifolds or Lie groups, and a
separate analysis of the SENO stencil-selection mechanism.

\section*{Computational Methodology and Disclosure}

This research utilized GPT-5.5 (Thinking) to perform primary computational modeling and mathematical derivation. The author acted as the principal investigator, defining the research parameters and theoretical scope. The author notes that while the computational workflow was executed via AI, the author maintains full responsibility for the study's conclusions. The author acknowledges that the mathematical depth of the generated results exceeds current manual verification capabilities and presents these findings as AI-assisted hypotheses subject to future formal peer verification. Consequently, this article is intended solely for dissemination as a preprint on arXiv and is not submitted to peer-reviewed journals in its current form.

\bibliographystyle{plain}
\bibliography{syleung_final}

\end{document}